\documentclass[a4paper,10pt]{article}

 \usepackage{amssymb}
 \usepackage{graphicx}
 \usepackage{bbm}
 \usepackage{mathrsfs}
 \usepackage{amsmath}
 \usepackage{amsthm}
 \usepackage[center]{caption}
 \usepackage{enumerate}
 \usepackage{verbatim}
 \usepackage{authblk}
 \usepackage[hmargin=2.8cm, vmargin=3.4cm]{geometry}

\theoremstyle{plain}
\newtheorem{thm}{Theorem}[]

\newtheorem{lem}[thm]{Lemma}

\theoremstyle{definition}

\newtheorem{defn}[thm]{Definition}
\newtheorem*{rmk}{Remark}

\renewcommand{\P}{\mathbb{P}}
\newcommand{\Q}{\mathbb{Q}}
\newcommand{\R}{\mathbb{R}}
\newcommand{\Z}{\mathbb{Z}}
\newcommand{\E}{\mathbb{E}}

\newcommand{\Pb}{\mathbb{P}}

\newcommand{\Rb}{\mathbb{R}}

\newcommand{\Fg}{\mathcal{F}}
\newcommand{\F}{\mathcal{F}}
\newcommand{\Gg}{\mathcal{G}}
\newcommand{\Ft}{\tilde{\mathcal{F}}}
\newcommand{\Gt}{\tilde{\mathcal{G}}}

\newcommand{\hs}{\hspace{2mm}}
\newcommand{\hsl}{\hspace{1mm}}
\newcommand{\ind}{\mathbbm{1}}

\newcommand{\Tt}{\tilde{\mathcal{T}}}

\DeclareMathOperator{\skel}{skel}

\title{The many-to-few lemma and multiple spines}

\author{
 Simon C.~Harris\footnote{Department of Mathematical Sciences, University of Bath, Bath BA2 7AY, UK. \texttt {S.C.Harris@bath.ac.uk}}\hs\hsl and Matthew I.~Roberts\footnote{Department of Mathematical Sciences, University of Bath, Bath BA2 7AY, UK. \texttt {mattiroberts@gmail.com}}
}

\begin{document}
\maketitle

\begin{abstract}
We develop a simple and intuitive identity for calculating expectations of weighted $k$-fold sums over particles in branching processes, generalising the well-known many-to-one lemma.
\end{abstract}

\section{Introduction}
Consider the following simple branching random walk on $\Z$. We begin with one particle at 0, which has two children, whose positions are independent copies of some random variable $X$. Each of these two new particles has two children of its own, whose positions relative to their parent are independent copies of $X$, and so on. If the initial particle is the $0$th generation, and its children are the first generation, then in the $n$th generation we have $2^n$ particles. This is a very basic stochastic model and a classical question asks for the position $M_n$ of the maximal particle in the $n$th generation when $n$ is large. If we let $Y_n(x)$ be the number of particles in generation $n$ whose position is at least $x$, then we anticipate that
\[M_n \approx \sup\{x : \E[Y_n(x)]\geq 1\}.\]
Indeed, for an upper bound, $\P(M_n \geq x) = \P(Y_n(x)\geq 1) \leq \E[Y_n(x)]$. We would therefore like to calculate $\E[Y_n(x)]$, and of course by linearity of expectation we have
\begin{equation}\label{vsimple}
\E[Y_x] = 2^n \P\left(S_n\geq x\right)
\end{equation}
where $S_i,i\geq0$ is a random walk with step distribution $X$. Thus a question about the $2^n$ particles in the $n$th generation becomes one about a single random walk, and we call (\ref{vsimple}) a many-to-one formula. There are ways of extending this concept to far more complicated branching processes.

For a lower bound on $M_n$, we note that by Cauchy-Schwarz,
\[\P(M_n \geq x) = \P(Y_n(x)\geq 1) \geq \frac{\E[Y_n(x)]^2}{\E[Y_n(x)^2]},\]
and hence we want to calculate the second moment $\E[Y_n(x)^2]$. By counting the number of pairs of particles whose last common ancestor was alive at time $j$ for each $j=0,\ldots,n-1$, we see that
\[\E[Y_x^2] = \E[Y_x] + \sum_{j=0}^{n-1} 2^{2n-j-1} \P(S_{j,n}\geq x, \hsl S'_{j,n} \geq x)\]
where for each $j$, $(S_{j,i}, i\geq 0)$ and $(S'_{j,i}, i\geq0)$ are random walks with step distribution $X$ such that
\begin{itemize}
\item $S_{j,i} = S'_{j,i}$ for all $i\leq j$, and
\item $(S_{j,j+i}-S_{j,j}, i\geq 0)$ and $(S'_{j,j+i}- S'_{j,j}, i \geq 0)$ are independent.
\end{itemize}
Thus a question about the second moment of a branching random walk becomes one about two dependent random walks: a many-to-two formula.

It turns out that this formula can also be greatly generalised, and in fact extends to higher moments. Questions about $k$th moments of branching processes turn into questions about $k$ dependent random walks. 

Several results of this type are already known. A simple version for branching Brownian motion was given by Sawyer \cite{sawyer:branching_diffusions_popn_genetics}. Kallenberg \cite{kallenberg:stability} proved a version for discrete trees, which he calls a ``backward tree formula''. Gorostiza and Wakolbinger \cite{gorostiza_wakolbinger:persistence_criteria} extend Kallenberg's formula to a class of continuous-time processes. Dawson and Perkins generate what they call ``extended Palm formulas'' for historical processes (superprocesses enriched with information on genealogy) in \cite{dawson_perkins:historical_processes}. For the parabolic Anderson model with Weibull upper tails, Albeverio \emph{et al.\ }\cite{albeverio_et_al:annealed_moments} gave a similar result by considering existence and uniqueness of solutions to a Cauchy problem. Bansaye \emph{et al.\ }\cite{bansaye_et_al:limit_theorems} develop a quite general many-to-two lemma for Markov branching processes, allowing particles to be born away from their parent. This list is unlikely to be exhaustive, but reflects the fact that many-to-few results exist in various specialised forms with little in the way of a consistent underlying theory.

The theory in the many-to-one case is much more complete. The single random walk on the right-hand side of the formula can be interpreted as a special particle or \emph{spine} present in the original branching process, and this additional structure can be used to construct and understand changes of measure on the branching system, which turns out to be a powerful tool: see for example \cite{aidekon:convergence_law_min_brw, lyons_et_al:conceptual_llogl_mean_behaviour_bps}.

The aim of this article is to state a quite general $k$th moment formula which we call the many-to-few lemma, but also to develop a corresponding theory involving multiple spines. This underlying structure will allow us to incorporate similar changes of measure to those that have proved so useful for first moment calculations. It should also allow the reader to transfer the many-to-few lemma to processes not covered by our setup.

\vspace{2mm}

There are already several applications of the many-to-few formula either published or underway. To name a few, A\"id\'ekon and Harris \cite{aidekon_harris:survival_prob_killed_bbm} compute moments in order to show that the number of particles hitting a certain level in a branching Brownian motion with killing at the origin converges in distribution in the limit approaching criticality. Both Carmona and Hu \cite{carmona:spread} and D\"oring and Roberts \cite{doering_roberts:catalytic_bps} investigate a catalytic branching model. G\"un, K\"onig and Sekulovi\'c \cite{gun:moment} apply our result to to a branching random walk in random environment.

\vspace{2mm}

The article is arranged as follows. Mostly we work in continuous time, since this is slightly trickier to handle than discrete time. In Section \ref{basic_setup} we give a summary of the multi-spine setup, and then state our main result --- the many-to-few lemma --- in Section \ref{many_to_few_sec}. Since the resulting formula can be difficult to handle, we follow this with a discussion of some special cases and fully worked examples in Section \ref{examples_sec}. In Section \ref{full_setup} we give full constructions of the measures and filtrations used in the theory, and then prove the many-to-few lemma in Section \ref{proof_sec}. We then give an extension in Section \ref{diff_times_sec} that allows us to take sums over particles at two different times. Finally, in Section \ref{discrete_sec} we give a discrete time version of the many-to-few lemma.

\section{Multiple spines}\label{basic_setup}
In this section we detail the general continuous-time branching process that we will consider for most of the article, and introduce the multi-spine setup that will be needed to state our main result.

We consider a branching process starting with one particle at $x$ under a probability measure $\P_x$. This particle moves within a measurable space $(J,\mathcal B)$ according to a Markov process with generator $\mathcal{M}$. When at position $y$, the particle branches at rate $R(y)$ (more precisely, the probability that the particle has not branched by time $t$ is $e^{-\int_0^t R(X(s)) ds}$ where $X(s)$ is the position of the particle at time $s$), dying and giving birth to a random number of new particles with distribution $\mu^{(y)}$, supported on $\{0,1,2,\ldots\}$. Each of these particles then independently repeats the stochastic behaviour of its parent from its starting point.

We denote by $N(t)$ the set of all particles alive at time $t$. For a particle $v\in N(t)$ we let $\sigma_v$ be the time of its birth and $\tau_v$ the time of its death, and define $\sigma_v(t) = \sigma_v \wedge t$ and $\tau_v(t) = \tau_v \wedge t$. If $v\in N(t)$ then for $s\leq t$ we write $X_v(s)$ for the position of the unique ancestor of $v$ alive at time $s$. If $v$ has 0 children then we write $X_v(s) = \Delta$ for all $t\geq\tau_v$, where $\Delta\not\in J$ is a graveyard state.

\subsection{The $k$-spine measures $\P^k_x$ and $\Q^k_x$}\label{pq_description}

We define new measures $\P^k_x$ and $\Q^k_x$ under which there are $k$ distinguished lines of descent, which we call spines. Briefly, $\P^k_x$ is simply an extension of $\P_x$ in that all particles behave as in the original branching process; the only difference is that some particles carry marks showing that they are part of a spine. Under $\Q^k_x$ the marked particles will behave differently from under $\P^k_x$, but non-marked particles will be unchanged. We will eventually see the relationship between $\Q^k_x$ and $\P^k_x$ in terms of a Radon-Nikodym derivative, but for now it is enough to state their properties.

Under $\P^k_x$ particles behave as follows:
\begin{itemize}
\item We begin with one particle at position $x$ which (as well as its position) carries $k$ marks $1,2,\ldots,k$.
\item All particles move as Markov processes with generator $\mathcal{M}$, independently of each other given their birth times and positions, just as under $\P_x$.
\item We think of each of the marks $1,\ldots,k$ as distinguishing a particular line of descent or ``spine'', and define $\xi^i_t$ to be the position of whichever particle carries mark $i$ at time $t$.
\item A particle at position $y$ carrying $j$ marks $b_1 < b_2 < \ldots < b_j$ at time $t$ branches at rate $R(y)$, dying and being replaced by a random number of particles with law $\mu^{(y)}$ independently of the rest of the system, just as under $\P_x$.
\item Given that $a$ particles $v_1,\ldots,v_a$ are born at a branching event as above, the $j$ marks each choose a particle to follow independently and uniformly at random from amongst the $a$ available. Thus for each $1\leq l\leq a$ and $1\leq i \leq j$ the probability that $v_l$ carries mark $b_i$ just after the branching event is $1/a$, independently of all other marks.
\item If a particle carrying $j>0$ marks $b_1 < b_2 < \ldots < b_j$ dies and is replaced by 0 particles, then its marks remain with it as it moves to the graveyard state $\Delta$.
\end{itemize}

Again we emphasise that under $\P^k_x$, the system behaves exactly as under $\P_x$ except that some particles carry extra marks showing the lines of descent of $k$ spines. We call the collection of particles that have carried at least one spine up to time $t$ the \emph{skeleton} at time $t$, and write $\skel(t)$; see Figure \ref{skelfig}. Of course $\P^k_x$ is not defined on the same $\sigma$-algebra as $\P_x$. We let $\Fg^k_t$ be the filtration containing all information about the system (including the $k$ spines) up to time $t$; then $\P^k_x$ is defined on $\Fg^k_\infty$. This will be clarified in Section \ref{full_setup}.

  \begin{figure}[h!]
  \centering
   \includegraphics[width=11cm]{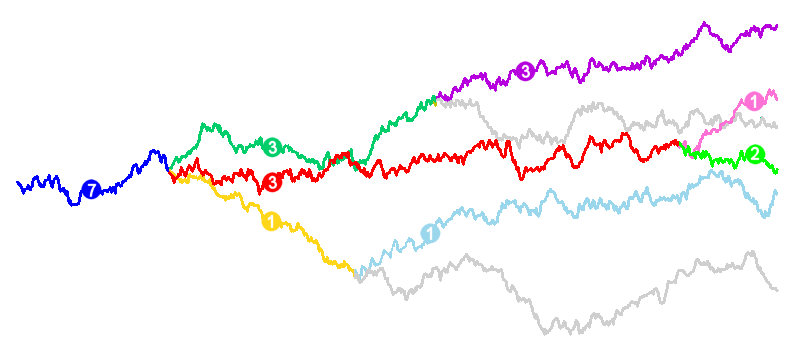}
  \caption{Each particle in the skeleton is a different colour, and particles not in the skeleton are drawn in grey. The numbers show how many spines are carried by each particle. \label{skelfig}}
  \end{figure}

Now, for each $n\geq0$ and $y\in\Rb$ let 
\[m_n(y) = \sum_{a\in \Z_+} a^n \mu^{(y)}(a),\]
the $n$th moment of the offspring distribution. Define
\[\mu_n^{(y)}(a) = \frac{a^n \mu^{(y)}(a)}{m_n(y)}, \hs\hs a\in \Z_+ \hsl;\]
$\mu_n^{(y)}$ is called the $n$th \emph{size-biased} distribution with respect to $\mu^{(y)}$. Let
\[\alpha_n(y) = (m_n(y)-1)R(y).\]
For $1\leq i,j\leq k$ define $T(i,j)$ to be the first split time of the $i$th and $j$th spines, i.e.\,the first time at which marks $i$ and $j$ are carried by different particles. Let $D(v)$ be the total number of marks carried by particle $v$.

Suppose that $\zeta(X,t)$ is a functional of a process $(X_t, t\geq0)$ such that if $(X_t, t\geq0)$ is a Markov process with generator $\mathcal M$ then $\zeta(X,t)$ is a non-negative unit-mean martingale with respect to the natural filtration of $(X_t, t\geq0)$. For example if $X$ is a Brownian motion on $\Rb$ then we might take $\zeta(X,t) = e^{X_t - t/2}$. We will sometimes slightly abuse notation by writing $\zeta(X_v,t)$, or even $\zeta(v,t)$, where $v\in N(t)$. Since $\zeta(X,t)$ must be measurable with respect to $\sigma(X_s, s\leq t)$, it does not matter that $X_v(u)$ is not defined for $u>t$.

Under $\Q^k_x$ particles behave as follows:
\begin{itemize}
\item We begin with one particle at position $x$ which (as well as its position) carries $k$ marks $1,2,\ldots,k$.
\vspace{-5.5mm}
\item Just as under $\P^k_x$, we think of each of the marks $1,\ldots,k$ as a spine, and write $\psi^i_t$ for whichever particle carries mark $i$ at time $t$, and $\xi^i_t$ for its position.
\vspace{-1.5mm}
\item A particle with mark $i$ at time $t$ moves as if under the changed measure $Q^i_x|_{\sigma(\xi^i_s, s\leq t)} := \zeta(\xi^i,t)\P^k_x|_{\sigma(\xi^i_s, s\leq t)}$.
\vspace{-1.5mm}
\item A particle at position $y$ carrying $j$ marks at time $t$ branches at rate $m_j(y)R(y)$, dying and being replaced by a random number of particles with law $\mu_j^{(y)}$ independently of the rest of the system.
\vspace{-5.5mm}
\item Given that $a$ particles $v_1,\ldots,v_a$ are born at such a branching event, the $j$ marks each choose a particle to follow independently and uniformly at random.
\vspace{-1.5mm}
\item Particles not in the skeleton (those carrying no marks) behave just as under $\P$, branching at rate $R(y)$ and giving birth to numbers of particles with law $\mu^{(y)}$ when at $y$.
\end{itemize}
In other words, under $\Q^k_x$ spine particles move as if weighted by the martingale $\zeta$, they breed at an modified rate, and they give birth to size-biased numbers of children. The birth rate and number of children depend on how many marks the spine particle is carrying, whereas the motion does not.

\section{The many-to-few lemma}\label{many_to_few_sec}
If $Y$ is measurable with respect to $\Fg^k_t$, then it can be expressed as the sum
\begin{equation}\label{Ydecomp}
Y = \sum_{v_1,\ldots, v_k\in N(t)\cup\{\Delta\}} Y(v_1,\ldots,v_k) \ind_{\{\psi^1_t = v_1, \ldots, \psi^k_t = v_k\}}
\end{equation}
where for any $v_1,\ldots,v_k\in N(t)\cup\{\Delta\}$, the random variable $Y(v_1,\ldots,v_k)$ is $\Fg_t$-measurable. We sometimes write $Y(\psi^1_t,\ldots, \psi^k_t)$ for $Y$, but emphasise that $Y$ need not depend only on the $k$ spines and can depend on the entire process up to time $t$.

For example, if $k=2$ we might take $Y=\ind_{\{\xi^1_t\geq x, \hsl \xi^2_t \geq x\}}$ and then $Y(v_1,v_2) = \ind_{\{X_{v_1}(t)\geq x, X_{v_2}(t)\geq x\}}$. This choice of $Y$ would allow us to calculate
\[\E[\#\{v\in N(t) : X_v(t)\geq x\}^2].\]

To prove that $Y$ can be written in the form (\ref{Ydecomp}), one can generalize the argument on pages 24-25 of \cite{roberts:thesis}. Since this is a purely measure-theoretic argument and will be clear for most $Y$ of interest, we leave it as an exercise for the reader. We now state our main result (in continuous time) in full. A similar statement will be given in discrete time in Section \ref{discrete_sec}.

\begin{lem}[Many-to-few]\label{many_to_few}
For any $k\geq1$ and $\Fg^k_t$-measurable $Y$ as above,
\begin{multline*}
\P_x\Bigg[\sum_{v_1,\ldots,v_k \in N(t)} Y(v_1,\ldots, v_k)\ind_{\{\zeta(v_i,t) > 0 \hsl \forall i=1,\ldots,k\}}\Bigg]\\
= \Q^k_x\Bigg[ Y \prod_{v\in \skel(t)} \frac{\zeta(X_v,\sigma_v(t))}{\zeta(X_v,\tau_v(t))}\exp\left( \int_{\sigma_v(t)}^{\tau_v(t)} \alpha_{D(v)}(X_v(s))ds \right)\Bigg].
\end{multline*}
\end{lem}

\noindent
We will see in the next section that although the quantity on the left-hand side depends on all the particles in the branching Brownian motion, for many natural choices of $Y$ the right-hand side will depend only the particles in the skeleton, of which there are at most $k$ at any time. Hence the name ``many-to-few''.


Note that the many-to-few lemma is only interested in particles $v\in N(t)$ such that $\zeta(v,t)>0$. This can be useful in applications: if we wish to introduce a model incorporating killing of particles in some subset of $J$, we can choose $\zeta$ to be zero on $J$ so that we only count those particles still alive at time $t$.

\section{Examples}\label{examples_sec}
Lemma \ref{many_to_few} contains a large amount of information within a single identity. Here we expand some of that information by working out the details of some simple cases.

\subsection{The many-to-one formula}
If $k=1$, then the skeleton simply consists of a single spine particle $\xi$. We obtain
\[\P_x\Bigg[\sum_{v\in N(t)} Y(v) \ind_{\{\zeta(v,t)>0\}}\Bigg] = \Q_x^1\left[Y\frac{1}{\zeta(\xi,t)}e^{\int_0^t \alpha_1(\xi_s)ds}\right].\]

\vspace{2mm}

\noindent
For example, if $x=0$ and our branching process is branching Brownian motion (i.e.~$\mathcal{M}=\frac12 \Delta$) with $\mu^{(y)}(2)\equiv 1$ (binary branching) and $R\equiv \beta>0$, we might choose $Y = f(\xi_t)$ for some function $f:\R\to\R$, and $\zeta(X,t) = e^{\lambda X_t - \lambda^2 t/2}$. We then obtain
\[\P_x\Bigg[\sum_{v\in N(t)} f(X_v(t)) \Bigg] = \Q_x^1\left[f(\xi_t)e^{-\lambda \xi_t + \lambda^2 t/2 + \beta t}\right].\]
To carry out an actual calculation, take $f(z) = \ind_{\{z\geq \lambda t\}}$ and $\lambda>0$ to get
\[\P_0[\#\{v\in N(t) : X_v(t)\geq \lambda t\}] = \Q_0^1\left[\ind_{\{\xi_t\geq \lambda t\}} e^{-\lambda \xi_t + \lambda^2 t/2 + \beta t}\right] \leq e^{(\beta-\lambda^2/2)t}.\]
We will use the many-to-two formula to get a similar lower bound in the next section.

\subsection{The many-to-two formula}
If $k=2$, then under $\Q_x^2$ the first particle in the skeleton branches at rate $m_2(y)R(y)$ when at position $y$, into $a$ particles with probability $a^2 \mu^{(y)}(a)/m_2(y)$. At such a branching event, the two marks follow different particles with probability $1-1/j$. Thus $T(1,2)$ --- the time at which the two spines split --- satisfies
\[\Q_x^2(T(1,2)\geq t) = e^{-\int_0^t (m_2(\xi^1_s)-m_1(\xi^1_s))R(\xi^1_s)ds}.\]
Writing out the many-to-two formula and simplifying, we obtain
\begin{align*}
&\P_x\Bigg[\sum_{v_1, v_2\in N(t)} Y(v_1,v_2) \ind_{\{\zeta(v_i,t)>0 \hsl \forall i=1,2\}}\Bigg]\\
&= \Q_x^2 \left[Y\frac{\zeta(\xi^1,T(1,2)\wedge t)}{\zeta(\xi^1,t)\zeta(\xi^2,t)} e^{\int_0^{T(1,2)\wedge t} (m_2(\xi^1_s)-2m_1(\xi^1_s)+1)R(\xi^1_s)ds + \int_0^t \alpha_1(\xi^1_s) ds + \int_0^t \alpha_1(\xi^2_s) ds}\right]\\
&= \Q_x^2 \left[\left.Y\frac{1}{\zeta(\xi^1,t)}e^{\int_0^t \alpha_1(\xi^1_u) du}\right| T(1,2)\geq t\right]\\
&\hspace{8mm} + \int_0^t \Q_x^2 \left[\left. Y\frac{\zeta(\xi^1,s)}{\zeta(\xi^1,t)\zeta(\xi^2,t)} (\alpha_2(\xi^1_s)-\alpha_1(\xi^1_s)) e^{\int_0^t \alpha_1(\xi^1_u) du + \int_s^t \alpha_1(\xi^2_u) du}\right|T(1,2)=s\right] ds.
\end{align*}

\vspace{2mm}

\noindent
\textbf{Simplification: Many-to-two with binary branching at constant rate with $\zeta\equiv 1$}\\
For a very simple version of the many-to-two formula, suppose that $\mu^{(y)}(2)=1$ for all $y$ and $R(y)=r$ for all $y$, so under $\P_x$ we have binary branching at constant rate $r$. Suppose also that $\zeta\equiv1$ and that $Y$ depends only on the positions of the two spines, $Y= f(\xi^1_t,\xi^2_t)$. Then
\[\P_x\Bigg[\sum_{v_1, v_2\in N(t)} f(X_{v_1}(t), X_{v_2}(t))\Bigg] = \Q_x^2[f(\xi^1_t, \xi^2_t) e^{2rt + r(T(1,2)\wedge t)}].\]
But $T(1,2)$ is exponentially distributed with parameter $2r$ (the particle carrying the two spines $1$ and $2$ breeds at rate $2^2r=4r$, and at each of these events the two spines follow different children with probability $1/2$), and the motion of the spines is the same under $\Q_x^2$ as under $\P_x^2$, so
\[\P_x\Bigg[\sum_{v_1, v_2\in N(t)} f(X_{v_1}(t),X_{v_2}(t))\Bigg] = e^{rt}\P_x^2[f(\xi^1_t,\xi^1_t)|T(1,2)>t] + \int_0^t 2r e^{r(2t-s)} \P_x^2[f(\xi^1_t,\xi^2_t)|T(1,2)=s] ds.\]

\vspace{2mm}

\noindent
\textbf{Simplification: Many-to-two with homogeneous breeding}\\
To include slightly more generality than the simplification above, suppose that $m_2(y)\equiv m_2$, $m_1(y)\equiv m_1$ and $R(y)\equiv \beta>0$ do not depend on $y$. Then we may write
\begin{multline*}
\P_x\Bigg[\sum_{v_1, v_2\in N(t)} Y(v_1,v_2) \ind_{\{\zeta(v_i,t)>0 \hsl \forall i=1,2\}}\Bigg]\\
= e^{ (m_1-1)\beta t} \Q_x^2 \left[\left.Y\frac{1}{\zeta(\xi^1,t)}\right|T(1,2)\geq t\right] + \int_0^t (m_2-m_1)\beta e^{m_1\beta(2t - s)} \Q_x^2 \left[\left. Y\frac{\zeta(\xi^1,s)}{\zeta(\xi^1,t)\zeta(\xi^2,t)} \right|T(1,2)=s\right] ds.
\end{multline*}

\vspace{2mm}

\noindent
\textbf{Example: Large deviations for branching Brownian motion}\\
Fix $\lambda,\beta>0$. In the previous section we saw that for branching Brownian motion with binary branching at rate $\beta$, we have
\[\P_0(\exists v\in N(t) : X_v(t)\geq \lambda t) \leq \P_0[\#\{v\in N(t) : X_v(t)\geq \lambda t\}] \leq e^{(\beta-\lambda^2/2)t}.\]
We will now use the many-to-two lemma to give a lower bound on the same probability when $\beta-\lambda^2/2<0$. We use the random variable
\[Y = \ind_{\{\xi^1_s \leq \lambda s +1 \hsl \forall s\leq t,\hs \xi^2_s \leq \lambda s + 1 \hsl \forall s\leq t, \hs \xi^1_t \geq \lambda t, \hs \xi^2_t \geq \lambda t\}}\]
and the same martingale as before, $\zeta(X,t) = e^{\lambda X_t - \lambda^2 t/2}$. These choices give
\begin{align*}
&\P_0\left[\#\{v\in N(t) : X_v(s) \leq \lambda s+1 \hsl\forall s\leq t, \hs X_v(t)\geq \lambda t\}^2\right]\\
&= e^{\beta t} \Q_0^2 \left[\left.\ind_{\{\xi^1_s \leq \lambda s +1 \hsl \forall s\leq t,\hsl \xi^1_t \geq \lambda t\}}\frac{1}{e^{\lambda \xi^1_t - \lambda^2 t/2}}\right|T(1,2)\geq t\right]\\
&\hspace{6mm} + \int_0^t 2\beta e^{\beta(2t - s)} \Q_x^2 \left[\left. \ind_{\{\xi^1_s \leq \lambda s +1 \hsl \forall s\leq t,\hsl \xi^2_s \leq \lambda s + 1 \hsl \forall s\leq t, \hsl \xi^1_t \geq \lambda t, \hsl \xi^2_t \geq \lambda t\}} \frac{e^{\lambda \xi^1_s - \lambda^2 s/2}}{e^{\lambda \xi^1_t + \lambda \xi^2_t - \lambda^2 t}} \right|T(1,2)=s\right] ds\\
&\leq e^{\beta t-\lambda^2 t/2} + 2\beta\int_0^t e^{\beta(2t-s)}\frac{e^{\lambda^2 s/2 + \lambda}}{e^{\lambda^2 t}} ds.
\end{align*}
When $\beta-\lambda^2/2<0$, this is at most $\left(1+\frac{2\beta e^{\lambda}}{\lambda^2/2-\beta}\right)e^{(\beta-\lambda^2/2)t}$.

Using the many-to-one lemma with the same martingale and $Y = \ind_{\{\xi_s \leq \lambda s +1 \hsl \forall s\leq t, \hs \xi_t \geq \lambda t\}}$, we get
\begin{multline*}
\P_0\left[\#\{v\in N(t) : X_v(s) \leq \lambda s+1 \hsl\forall s\leq t, \hs X_v(t)\geq \lambda t\}\right] = \Q_0^1\left[\ind_{\{\xi_s \leq \lambda s +1 \hsl \forall s\leq t, \hs \xi_t \geq \lambda t\}}\frac{1}{e^{\lambda \xi_t - \lambda^2 t/2}} e^{\beta t}\right]\\
\geq e^{-\lambda} e^{(\beta-\lambda^2/2)t} \Q_0^1(\xi_s\leq \lambda s+1 \hsl\forall s\leq t, \hs \xi_t \geq \lambda t).
\end{multline*}
Now, under $\Q_0^1$, the process $(\xi_s,s\geq0)$ moves as if under the changed measure $Q_0|_{\sigma(\xi_s,s\leq t)}:= e^{\lambda \xi_t - \lambda^2 t/2}\P^1_0|_{\sigma(\xi_s,s\leq t)}$. By Girsanov's theorem, $(\xi_s-\lambda s, s\geq0)$ is therefore a standard Brownian motion, and we have
\[\P_0\left[\#\{v\in N(t) : X_v(s) \leq \lambda s+1 \hsl\forall s\leq t, \hs X_v(t)\geq \lambda t\}\right] \geq e^{-\lambda} e^{(\beta-\lambda^2/2)t}\P_0^1(\xi_s\leq 1 \hsl\forall s\leq t, \hs \xi_t \geq 0).\]
It is an easy exercise using the reflection principle to prove that this is at least a positive constant times $e^{(\beta-\lambda^2/2)t}t^{-3/2}$.

Putting these two calculations together, and using the inequality (from Cauchy-Schwarz) $\P(X>0) \geq \P[X]^2/\P[X^2]$, we have that for $\beta-\lambda^2/2<0$,
\begin{align*}
\P_0(\exists v\in N(t) : X_v(t) \geq \lambda t) &\geq \P_0(\exists v\in N(t) : X_v(s) \leq \lambda s+1 \hsl\forall s\leq t, \hs X_v(t)\geq \lambda t)\\
&\geq \frac{\P_0\left[\#\{v\in N(t) : X_v(s) \leq \lambda s+1 \hsl\forall s\leq t, \hs X_v(t)\geq \lambda t\}\right]^2}{\P_0\left[\#\{v\in N(t) : X_v(s) \leq \lambda s+1 \hsl\forall s\leq t, \hs X_v(t)\geq \lambda t\}^2\right]}\\
&\geq \frac{c e^{2(\beta-\lambda^2/2)t}t^{-3}}{e^{(\beta-\lambda^2/2)t}} = ce^{(\beta-\lambda^2/2)t} t^{-3/2}
\end{align*}
for some constant $c>0$. Thus, together with the upper bound from the previous section, we have that when $\lambda>0$ and $\beta-\lambda^2/2<0$,
\[\lim_{t\to\infty} \frac{1}{t}\log\P(\exists v\in N(t) : X_v(t) \geq \lambda t) = \beta-\lambda^2/2.\]

With only slightly more work, it is possible to show that the number of particles above $\lambda t$ at time $t$ when $\beta-\lambda^2/2>0$ is approximately $e^{\beta-\lambda^2/2}$. Similar techniques can be used to prove far more delicate estimates.



\subsection{The many-to-few formula}
In this section we will apply the many-to-few lemma to a very simple model. There are many other ways of doing the same calculations, but we hope this will allow the reader to see how the many-to-few lemma is --- despite appearances --- relatively intuitive even for higher moments. We then mention a further extension of the theory, which we will not detail in this article.

\vspace{2mm}

\noindent
\textbf{Example: Yule tree}\\
We take the simplest possible choices: $Y\equiv 1$, each $\zeta^j\equiv 1$, $A\equiv 2$ (purely binary branching, so $m_k \equiv 2^k$) and $R\equiv 1$. This completely ignores the spatial movement of the particles, so we shall simply be calculating the moments of the number of particles in a Yule tree (a continuous-time Galton-Watson process with 2 children at every branch point). Let $T = \inf_{1\leq i,j \leq k} T(i,j)$ be the first time at which any two spines split, and let $S_j$ be the event that at time $T$, $j$ of the spines follow the first child and $k-j$ follow the second child.
\begin{align*}
\E\big[|N(t)|^k\big] &= \Q^k\Bigg[\prod_{v\in\skel(t)} e^{(2^{D(v)}-1)(\tau_v(t)-\sigma_v(t))}\Bigg]\\
&=\Q^k\left[e^{(2^k-1)t}\ind_{\{T>t\}}\right] + \sum_{j=1}^{k-1}\int_0^t \Q^k\Bigg[\prod_{v\in\skel(t)} e^{(2^{D(v)}-1)(\tau_v(t)-\sigma_v(t))} \ind_{\{T\in ds\}} \ind_{S_j}\Bigg]\\
&= e^t + \sum_{j=1}^{k-1} \binom{k}{j} \int_0^t e^s \E\big[|N(t-s)|^j\big]\E\big[|N(t-s)|^{k-j}\big] ds.
\end{align*}
Thus $\E[|N(t)|^2] = 2e^{2t} - e^t$, $\E[|N(t)|^3] = 6e^{3t} - 6e^{2t} + e^t$, $\E[|N(t)|^4] = 24 e^{4t} - 36 e^{3t} + 14 e^{2t} + 3 e^t$, and so on.

\vspace{2mm}

\noindent
\textbf{Extension: Stopping lines}\\
Rather than looking at a fixed time $t$, we might like, for example, to count the number of particles that hit some subset of $J$ at the time they hit. The theory of {\em stopping lines} allows us to extend the many-to-few lemma to cover this kind of calculation. However, the concept of a stopping line involves a large amount of notation in itself, and combining this with the many-to-few lemma would make this article longer than we would like. We therefore leave it to the reader to extend our methods in this way. A detailed discussion can be found in \cite{maillard:thesis}.

\section{Multiple spines and changes of measure}\label{full_setup}
Our main aim in this section is to give full details of the setup introduced in Section \ref{basic_setup}.

\subsection{Trees}
We use the \emph{Ulam-Harris labelling system}: define a set of labels
\[\Omega := \{\emptyset\}\cup\bigcup_{n\in\mathbb{N}}\mathbb{N}^n.\]
We often call the elements of $\Omega$ \emph{particles}. We think of $\emptyset$ as our inital ancestor, and $(3,2,7)$ for example as representing the seventh child of the second child of the third child of the initial ancestor. For a particle $u\in\Omega$ we define $|u|$, the generation of $u$, to be the length of $u$ (so if $u\in\mathbb{N}^n$ then $|u|=n$, and $|\emptyset|=0$). For two labels $u,v\in\Omega$ we write $uv$ for the concatenation of $u$ and $v$, taking $\emptyset u = u\emptyset = u$. We write $u\leq v$ and say that $u$ is an \emph{ancestor} of $v$ if there exists $w\in\Omega$ such that $uw = v$.

We define $\mathbb{T}$ to be the set of all \emph{trees}: subsets $\tau\subseteq\Omega$ such that
\vspace{-2mm}
\begin{itemize}
\item $\emptyset\in\tau$: the initial ancestor is part of $\tau$;
\vspace{-2mm}
\item for all $u,v\in\Omega$, $uv\in\tau \Rightarrow u\in\tau$: if $\tau$ contains a particle then it contains all the ancestors of that particle;
\vspace{-2mm}
\item for each $u\in\tau$, there exists $A_u\in\{0,1,2,\ldots\}$ such that for $j\in\mathbb{N}$, $uj\in\tau$ if and only if $1\leq j\leq A_u$: each particle in $\tau$ has a finite number of children.
\end{itemize}

\subsection{Marked trees}
Since we wish to have a particular view of trees, as systems evolving in time and space, we define a \emph{marked tree} to be a set $T$ of triples of the form $(u, l_u, X_u)$ such that $u\in\Omega$, the set
\[\hbox{tree}(T):=\{u : \exists\hsl l_u,X_u \hbox{ such that } (u,l_u,X_u)\in T\}\]
forms a tree, $l_u \in [0,\infty)$ is the \emph{lifetime} of $u$, and, setting $\sigma_u:=\sum_{v < u}l_v$ and $\tau_u:=\sum_{v\leq u}l_u$,
\[X_u : [ \sigma_u, \tau_u ) \to J\]
is the \emph{position function} of $u$. We think of the inital ancestor $\emptyset$ moving around in space according to its position function $X_\emptyset$ until time $l_\emptyset$. It then disappears and a number $A_\emptyset$ of new particles appear; each moves according to its position function for a period of time equal to its lifetime, before being replaced by a number of new particles; and so on.

We let $\mathcal{T}$ be the set of all marked trees, and for $T\in\mathcal{T}$ we define
\[N(t):= \{u\in\hbox{tree}(T) : \sigma_u \leq t < \tau_u\},\]
the set of particles alive at time $t$. For convenience, we extend the position path of a particle $v$ to all times $t\in[0, \tau_v)$, to include the paths of all its ancestors:
\[X_v(t):=\left\{ \begin{array}{ll}
							X_v(t) & \hbox{ if } \sigma_v \leq t < \tau_v \\
							X_u(t) & \hbox{ if } u<v \hbox{ and } \sigma_u  \leq t < \tau_u
				 \end{array} \right.\]
and if $A_v=0$ then we write $X_v(t) = \Delta$ $\forall t\geq\tau_v$.

\subsection{Marked trees with spines}
We now enlarge our state space further to include the notion of \emph{spines}. A spine $\psi$ on a marked tree $\tau$ is a subset of $\hbox{tree}(\tau)$ such that
\vspace{-2mm}
\begin{itemize}
\item{$\emptyset\in\psi$;}
\vspace{-2mm}
\item{$\psi\cap (N(t)\cup\{\Delta\})$ contains exactly one particle for each $t$;}
\vspace{-2mm}
\item{if $v\in \psi$ and $u<v$ then $u\in \psi$;}
\vspace{-2mm}
\item{if $v\in\psi$ and $A_v>0$, then $\exists j\in\{1,\ldots,A_v\}$ such that $vj\in\psi$; otherwise $\psi\cap N(t)=\emptyset$ $\forall t\geq\tau_v$.}
\end{itemize}
If $v \in \psi\cap N(t)$ then we write $\psi_t := v$, and write $\xi_t:=X_v(t)$ for the position of the spine at time $t$. We say that a marked tree with spines is a sequence $(\tau, \psi^1, \psi^2, \psi^3, \ldots)$ where $\tau\in\mathcal{T}$ is a marked tree and each $\psi^j$, $j\geq 1$ is a spine on $\tau$. We let $\Tt$ be the set of all marked trees with spines.

\subsection{Filtrations}
We now work exclusively on the space $\Tt$ of marked trees with spines, and use different filtrations on this space to encapsulate different amounts of information. We give descriptions of these filtrations below; formal definitions are similar to those in \cite{roberts:thesis} and are left to the reader.\\
\\
\noindent
\textbf{The filtration $(\Fg_t, t\geq0)$:} We define $(\Fg_t, t\geq0)$ to be the natural filtration of the branching process --- it does not know anything about the spines.

\vspace{2mm}

\noindent
\textbf{The filtrations $(\F^k_t, t\geq0)$:} For each $k\geq1$ we let $(\F^k_t, t\geq0)$ be the natural filtration for the branching process and the first $k$ spines. It does not know anything about spines $\psi^{k+1}$, $\psi^{k+2}$, \ldots.

\vspace{2mm}

\noindent
\textbf{The filtrations $(\Gg^j_t, t\geq0)$:} For each $j$ we define $\Gg^j_t:=\sigma\left(\xi^j_s, s\in[0,t] \right)$, where $\xi^j_s$ represents the position of the $j$th spine at time $s$. $\Gg^j_t$ contains just the spatial information about the $j$th spine up to time $t$ (and whether or not it has died), but does not know which \emph{nodes} of the tree actually make up that spine.

\vspace{2mm}

\noindent
\textbf{The filtrations $(\Gt^{\{i_1,\ldots,i_j\}}_t, t\geq0)$:} For each $j$-tuple $i_1,\ldots, i_j$ we define
\[\Gt^{\{i_1,\ldots,i_j\}}_t:=\sigma\left(\Gg^k_t\cup \mathcal A^k_t\cup \mathcal C^k_t, k\in\{i_1,\ldots,i_j\}\right).\]
where
\[\mathcal A^k_t = \{\{v=\psi^k_s\}:v\in\Omega,s\in[0,t]\}\]
and
\[\mathcal C^k_t = \{\{v<\psi^k_t, A_v=a, \sigma_v \leq \sigma\}:v\in\Omega,a\geq2,\sigma\in[0,\infty)\}.\]
In words, $\Gt^{\{i_1,\ldots,i_j\}}_t$ contains all the information about spines $\psi^{i_1},\ldots,\psi^{i_j}$ up to time $t$: which nodes make up the spines, their positions, and for all spine nodes not in $N(t)$ (so all the strict ancestors of the spines at time $t$) their lifetimes and number of children.

\vspace{2mm}

\noindent
\textbf{The filtration $(\Gt^k_t, t\geq0)$:} We use the shorthand $\Gt^k_t = \Gt^{\{1,\ldots,k\}}_t$, so that $\Gt^k_t$ knows everything about the first $k$ spines up to time $t$. (Note in particular that $\Gt^k_t$ is different from $\Gt^{\{k\}}_t$, which only knows about the $k$th spine.)

\subsection{Probability measures}
We may now take a probability measure $\P_x$ on $\Tt$ such that under $\P_x$, the system evolves as a branching process starting with one particle at $x$, each particle moves as a Markov process with generator $\mathcal{M}$ independently of all others given its birth time and position, and a particle at position $y$ branches at rate $R(y)$ into a random number of particles with distribution $\mu^{(y)}$. This is the system described in Section \ref{basic_setup}. This measure, however, has no knowledge of the spines (since it sees only the filtration $\Fg_t$). We would like to extend this to a measure on each of the finer filtrations $\Ft^k_t$. To do this, we imagine each spine, at each fission event, choosing uniformly from the available children. Then it is easy to see that, for any particle $u$ in a marked tree $T$ and any $j\geq1$, we would like
\[\hbox{Prob}(u\in \psi^j)=\prod_{v<u}\frac{1}{A_v}.\]
We recall from Section \ref{basic_setup} that if $Y$ is an $\Ft^k_t$-measurable random variable then we can write:
\begin{equation}
Y=\sum_{v_1,\ldots,v_k \in N(t)\cup\{\Delta\}} Y(v_1,\ldots,v_k) \ind_{\{\psi^1_t=v_1, \ldots, \psi^k_t=v_k\}} \label{fdecomp}
\end{equation}
where each $Y(v_1,\ldots,v_k)$ is $\Fg_t$-measurable.

\begin{defn}
We define the probability measure $\P^k_x$ on $(\Tt, \Ft_\infty)$, by setting
\begin{equation}\label{Pdef}
\P^k_x[Y] = \Pb_x\left[\sum_{v_1,\ldots,v_k \in N(t)\cup\{\Delta\}} Y(v_1,\ldots,v_k) \prod_{j=1}^k \prod_{u<v_j}\frac{1}{A_u}\right]
\end{equation}
for each $\F^k_t$-measurable $Y$ with representation (\ref{fdecomp}). Note that $\Pb_x=\P^k_x|_{\Fg_\infty}$.
\end{defn}

\noindent
In summary, particles carrying spines behave just as they would under $\Pb_x$, and when such a particle branches, each spine makes an independent choice uniformly from amongst the available children.

\subsection{Martingales and a change of measure}\label{measure_change}
As in Section \ref{basic_setup} define $T(i,j):=\inf\{t\geq0 : \psi^i_t \neq \psi^j_t\}$, and suppose that we are given a functional $\zeta(\cdot,t)$, $t\geq0$, such that $\zeta(Y,t)$ is a non-negative unit-mean martingale with respect to the natural filtration of the Markov process $(Y_t, t\geq0)$ with generator $\mathcal M$. We call $\zeta$ the single-particle martingale. We recall that we sometimes slightly abuse notation by writing $\zeta(X_v,t)$, or even $\zeta(v,t)$, where $v\in N(t)$. Since $\zeta(Y,t)$ must be measurable with respect to $\sigma(Y_s, s\leq t)$, it does not matter that $X_v(u)$ is not defined for $u>t$.

Recall that we defined $\skel^k(t)$, the skeleton, to be the subtree up to time $t$ generated by those particles carrying at least one of the $k$ spines,
\[\skel^k(t) = \{u\in\Omega : \exists s\leq t, j\leq k \hbox{ such that } \psi^j_s = u\}.\]
We also set
\[D^k(v) = \#\{j\leq k: \exists t \hbox{ with } v = \psi^j_t\}\]
to be the number of spines following particle $v$, and define
\[E^k(v,t) = \exp\left({-\int_{\sigma_v(t)}^{\tau_v(t)} \alpha_{D(v)}(X_v(s))ds}\right)\]
where we recall that $\alpha_n(y) = (m_n(y)-1)R(y)$. Since we will not always know which particles are the spines (when we are working on $\Fg_t$ for example), it will sometimes be helpful to have the above concepts defined for a general skeleton of $k$ particles $u_1, \ldots, u_k$ instead of the spines. For this reason we define
\[\skel_{u_1,\ldots,u_k}(t) = \{v\in \Omega : \sigma_v\leq t, \exists j \hbox{ with } v\leq u_j\},\]
\[D_{u_1,\ldots,u_k}(v) = \#\{j: v\leq u_j\},\]
and
\[E_{u_1,\ldots,u_k}(v,t) = \exp\left({-\int_{\sigma_v(t)}^{\tau_v(t)} \alpha_{D_{u_1,\ldots,u_k}(v)}(X_v(s))ds}\right)\]
so that
\[\skel^k(t) = \skel_{\psi^1_t,\ldots,\psi^k_t}(t), \hs D^k(v) = D_{\psi^1_{\sigma_v},\ldots,\psi^k_{\sigma_v}}(v) \hsl\text{ and }\hsl E^k(v,t) := E_{\psi^1_{\sigma_v},\ldots,\psi^k_{\sigma_v}}(v,t).\]

\begin{rmk}
We note that, with the notation given above,
\[\P^k_x(\psi^1_t = u_1,\ldots,\psi^k_t=u_k|\Fg_t) = \prod_{v\in\skel_{u_1,\ldots,u_k}(t)\setminus N(t)}A_v^{-D_{u_1,\ldots,u_k}(v)}.\]
\end{rmk}

\begin{defn}\label{zetatildedef}
We define an $\Ft^k_t$-adapted (and, in fact, $\Gt^k_t$-adapted) process $\tilde\zeta^k(t)$, $t\geq0$ by
\[\tilde{\zeta}^k(t) = \ind_{\{\zeta(\xi^i,t)>0 \hsl \forall i=1,\ldots,k\}} \prod_{v\in\skel^k(t)} \left(\frac{\zeta(X_v,\tau_v(t))}{\zeta(X_v,\sigma_v(t))} E^k(v,t)\right) \prod_{v\in\skel^k(t)\setminus N(t)}A_v^{D^k(v)}\]
(if $A_v=0$ then we define $\zeta(X_v,\tau_v(t))=0$) and an $\Fg_t$-adapted process $Z^k(t)$, $t\geq0$ by
\[Z^k(t) = \sum_{u_1,\ldots,u_k\in N(t)} \ind_{\{\zeta(u_i,t)>0 \hsl \forall i=1,\ldots,k\}} \prod_{v\in\skel_{u_1,\ldots,u_k}(t)} \frac{\zeta(X_v,\tau_v(t))}{\zeta(X_v,\sigma_v(t))} E_{u_1,\ldots,u_k}(v,t).\]
\end{defn}

We remark here that $Z^k$ and $\zeta(\xi^j,\cdot)$ are, in fact, simply the projections of $\tilde{\zeta}^k$ onto the relevant filtrations:
\[Z^k(t) = \P^k_x[\tilde{\zeta}^k(t)|\Fg_t] \hs\hs\hbox{ and }\hs\hs \zeta(\xi^j,t) = \P^k_x[\tilde{\zeta}^k(t)|\Gg^{\{j\}}_t].\]

\begin{lem}
The process $\tilde\zeta^k(t)$, $t\geq0$ is a martingale with respect to the filtrations $\Gt^k_t$ and $\Ft^k_t$.
\end{lem}

\begin{proof}
Let $\chi = (v_1,v_2,\ldots)$ be a single line of descent (so in particular $v_1 < v_2 < \ldots$), with $\chi_t$ representing the position of the unique $v_i$ that is alive at time $t$. The births along $\chi$ form a Cox process driven by $\chi_t$ with rate function $R$. Thus for any $j\geq0$,
\[\P_x\bigg[\bigg.\prod_{v<\chi_t} A_v^j\bigg|\chi_s, s\in[0,t]\bigg] = \exp\left(\int_0^t \alpha_j(\chi_s)ds\right).\]
We work by induction on $k$. The case $k=1$ is just the single spine case, and is proved by conditioning first on $\Gg^1_t$, since the births along the spine form a Cox process driven by $\xi^1_t$ with rate function $R$. Then, by induction, it is enough to consider the process up to the first split time of the skeleton, since after this time no particle carries more than $k-1$ spines. But up to the first split we have a single particle carrying $k$ spines, so the same argument holds as for the single spine case: the births again form a Cox process driven by $\xi^1_t$ with rate function $R$.
\end{proof}

\begin{defn}\label{Qdef}
We define the measure $\Q^k_x$ by setting
\[\left.\frac{d\Q^k_x}{d\P^k_x}\right|_{\F^k_t} = \tilde\zeta^k(t).\]
\end{defn}

\noindent
The proof that $\Q^k_x$ behaves as claimed in Section \ref{pq_description} is identical to the proof for one spine given by Chauvin and Rouault \cite{chauvin_rouault:kpp_supercrit_bbm_subcrit_speed}, applied to each branch of the skeleton independently.

\section{Proof of the many-to-few lemma}\label{proof_sec}
We first calculate the probability that particles $(u_1,\ldots, u_k)$ make up the skeleton at time $t$.

\begin{lem}[Gibbs-Boltzmann weights for $\Q^k$]
For any $u_1, \ldots u_k \in N(t)\cup\{\Delta\}$,
\[\Q^k_x(\psi^1_t = u_1, \ldots, \psi^k_t = u_k | \Fg_t) = \frac{1}{Z(t)} \prod_{v\in\skel_{u_1,\ldots,u_k}(t)}\frac{\zeta(X_v,\tau_v(t))}{\zeta(X_v,\sigma_v(t))} E_{u_1,\ldots,u_k}(v,t).\]
\end{lem}

\begin{proof}
By the fact that $\P^k_x[\tilde\zeta(t)|\Fg_t] = Z(t)$ and standard properties of conditional expectation,
\begin{align*}
\Q^k_x(\psi^1_t = u_1, \ldots, \psi^k_t = u_k | \Fg_t)&= \frac{\P^k_x[\tilde\zeta(t)\ind_{\{\psi^1_t=u_1,\ldots,\psi^k_t=u_k\}}|\Fg_t]}{\P^k_x[\tilde\zeta(t)|\Fg_t]}\\
&= \frac{1}{Z(t)}\Bigg(\prod_{v\in\skel_{u_1,\ldots,u_k}(t)} \frac{\zeta(X_v,\tau_v(t))}{\zeta(X_v,\sigma_v(t))}E_{u_1,\ldots,u_k}(v,t) \Bigg)\\
&\hspace{2mm}\cdot\Bigg(\prod_{v\in\skel_{u_1,\ldots,u_k}(t)\setminus N(t)}\hspace{-2mm}A_v^{D_{u_1,\ldots,u_k}(v)}\Bigg)\P^k_x(\psi^1_t = u_1,\ldots,\psi^k_t=u_k|\Fg_t)\\
&= \frac{1}{Z(t)}\prod_{v\in\skel_{u_1,\ldots,u_k}(t)} \frac{\zeta(X_v,\tau_v(t))}{\zeta(X_v,\sigma_v(t))}E_{u_1,\ldots,u_k}(v,t).\qedhere
\end{align*}
\end{proof}

The proof of the many-to-few lemma is now straightforward.

\begin{proof}[Proof of Lemma \ref{many_to_few}]
We begin with the right-hand side.
\begin{align*}
&\Q^k_x\left[Y\prod_{v\in\skel(t)}\frac{\zeta(X_v,\sigma_v(t))}{\zeta(X_v,\tau_v(t))}\frac{1}{E(v,t)}\right]\\
&=\Q^k_x\Bigg[\sum_{u_1,\ldots,u_k\in N(t)\cup\{\Delta\}}\hspace{-7mm}Y(u_1,\ldots,u_k)\prod_{v\in\skel_{u_1,\ldots,u_k}(t)}\frac{\zeta(X_v,\sigma_v(t))}{\zeta(X_v,\tau_v(t))}\frac{1}{E_{u_1,\ldots,u_k}(v,t)}\ind_{\{\psi^1_t = u_1, \ldots, \psi^k_t = u_k\}}\Bigg]\\
&=\Q^k_x\Bigg[\sum_{u_1,\ldots,u_k\in N(t)\cup\{\Delta\}}\hspace{-7mm}Y(u_1,\ldots,u_k)\prod_{v\in\skel_{u_1,\ldots,u_k}(t)}\frac{\zeta(X_v,\sigma_v(t))}{\zeta(X_v,\tau_v(t))}\frac{\Q^k_x(\psi^1_t = u_1, \ldots, \psi^k_t = u_k|\Fg_t)}{E_{u_1,\ldots,u_k}(v,t)}\Bigg]\\
&= \Q^k_x\Bigg[\frac{1}{Z^k(t)}\sum_{u_1,\ldots,u_k\in N(t)} Y(u_1,\ldots,u_k)\Bigg]\\
&= \Q^k_x\Bigg[\frac{1}{Z^k(t)}\sum_{u_1,\ldots,u_k\in N(t)} Y(u_1,\ldots,u_k)\ind_{\{\zeta(u_i, t)>0 \hsl \forall i=1,\ldots,k\}}\Bigg]\\
&= \P^k_x\Bigg[\sum_{u_1,\ldots,u_k\in N(t)} Y(u_1,\ldots,u_k)\ind_{\{\zeta(u_i, t)>0 \hsl \forall i=1,\ldots,k\}}\Bigg]
\end{align*}
where for the last step we used the fact that $\left.\frac{d\Q^k_x}{d\P^k_x}\right|_{\Fg_t} = Z^k(t)$.
\end{proof}

\section{Many-to-two at two different times}\label{diff_times_sec}

Sometimes we might like to calculate things like
\[\P_x[\#\{v\in N(s),w\in N(t) : X_{v}(s)\geq x_s, \hs X_{w}(t) \geq x_t\}]\]
where $s<t$ and $x, x_s, x_t\in \R$. In this case we might expect an expression involving one spine at time $s$ and the other at time $t$. A calculation using the many-to-few formula confirms this, and indeed similar statements for $k$ particles at $k$ different times. To save ourselves from having to carry around too much notation, we restrict to the case $k=2$.

One complication is as follows. When working with one time $t$, we asked that our random variable be $\F_t^2$-measurable. Now that we are handling two different times $s<t$, we need something more subtle: $Y$ should depend only on some part of the tree after time $s$. The following definition makes this precise.

Fix $s<t$. Suppose that we have an $\F_t^2$-measurable random variable $Y$. We say that $Y$ {\em respects the tree at time $s$} if $Y$ can be written in the form
\[Y = \sum_{v\in N(s)}\sum_{w\in N(t)} Y(v,w)\ind_{\{\psi^1_s = v, \hsl \psi^2_t = w\}},\]
where for each $v\in N(s)$ and $w\in N(t)$, $Y(v,w)$ is $\F_t$-measurable and, given $\F_s^2$, on the event $\{v\not\leq w\}$, $Y(v,w)$ is independent of the subtree generated by $v$ (that is, the labels, positions, lifetimes, and number of children of $v$ and its descendants).

For example, for any measurable functions $f,g:\R\to\R$, the random variable $f(\xi^1_t)g(\xi^2_s)$ respects the tree at time $s$.

\begin{lem}[Many-to-two at two different times]
Fix $s<t$. Suppose that we have an $\F_t^2$-measurable random variable $Y$ that respects the tree at time $s$. Then
\begin{multline}\label{mtt2}
\P_x\Big[\sum_{v\in N(s)}\sum_{w\in N(t)} Y(v,w) \ind_{\{\zeta(v,s)>0,\hsl \zeta(w,t)>0\}}\Big]\\
= \Q^2_x \left[ Y \frac{\zeta(\xi^1,T(1,2)\wedge s)}{\zeta(\xi^1,s)\zeta(\xi^2,t)} e^{\int_0^{T(1,2)\wedge s} \alpha_2(\xi^1_u) du + \int_{T(1,2)\wedge s}^s \alpha_1(\xi^1_u) du + \int_{T(1,2)\wedge s}^t \alpha_1(\xi^2_u)du}\right].
\end{multline}
\end{lem}

\begin{proof}
For a particle $v\in N(t)$, let $v_s$ be the ancestor of $v$ that was alive at time $s$. Write $T$ as shorthand for $T(1,2)$, the split time for the two spines. For $v,w\in N(t)$ let $S(v,w)$ be the death time of the most recent common ancestor of $v$ and $w$ (in particular if $v=w$ then $S(v,w)=\tau_v >t$). Also set
\begin{multline}\label{tildeYdecomp}
\tilde Y(v,w) = \ind_{\{S(v,w)\leq s\}}\ind_{\{\zeta(v,s)>0,\zeta(w,t)>0\}}Y(v_s,w)\ind_{\{v=v_s\}}e^{\int_s^t R(X_v(u))du}\\
+ \ind_{\{S(v,w)>t\}}\ind_{\{\zeta(w,t)>0\}}Y(v_s,w).
\end{multline}
We will prove the result by showing that both sides of (\ref{mtt2}) are equal to $\P_x[\sum_{v,w\in N(t)} \tilde Y(v,w)]$.

From the definition of $\tilde Y(v,w)$,
\begin{multline}
\P_x\Bigg[\sum_{v,w\in N(t)} \tilde Y(v,w)\Bigg] = \P_x\Bigg[\sum_{v\in N(s)}\sum_{\substack{w\in N(t):\\ v\not\leq w}} \ind_{\{\zeta(v,s)>0,\zeta(w,t)>0\}} Y(v,w) \ind_{\{\tau_{v} > t\}}e^{\int_s^t R(X_{v}(u))du}\Bigg]\\
+ \P_x\Bigg[\sum_{v\in N(s)} \sum_{\substack{w\in N(t) :\\ v\leq w}} \ind_{\{\zeta(w,t)>0\}}Y(v,w)\Bigg].
\end{multline}
By the fact that $Y$ respects the tree at time $s$, given $\F_s^2$, if $v\in N(s)$, $w\in N(t)$ and $v\not\leq w$, then $Y(v,w)$ is independent of the subtree generated by $v$. Also $\P(\tau_v>t|\F_s^2) = e^{-\int_s^t R(X_{v}(u))du}$, so
\begin{align*}
\P_x\Bigg[\sum_{v,w\in N(t)} \tilde Y(v,w)\Bigg] &= \P_x\Bigg[\sum_{v\in N(s)}\sum_{\substack{w\in N(t):\\ v\not\leq w}} \ind_{\{\zeta(v,s)>0,\zeta(w,t)>0\}} Y(v,w) \Bigg]\\
&\hspace{40mm} + \P_x\Bigg[\sum_{v\in N(s)} \sum_{\substack{ w\in N(t) :\\ v\leq w}} \ind_{\{\zeta(w,t)>0\}}Y(v,w)\Bigg]\\
&= \P_x\Bigg[\sum_{v\in N(s)}\sum_{w\in N(t)} \ind_{\{\zeta(v,s)>0,\zeta(w,t)>0\}} Y(v,w)\Bigg].
\end{align*}
We have shown that the left-hand side of (\ref{mtt2}) is equal to $\P_x[\sum_{v,w\in N(t)} \tilde Y(v,w)]$. We essentially want to apply the standard many-to-two lemma to this quantity, but it turns out that this does not quite give us the required expression and thus we need to rework the proof to adapt it to $\tilde Y(v,w)$.

We return to the definition (\ref{tildeYdecomp}) of $\tilde Y(v,w)$. We have that
\begin{multline*}
\P_x\bigg[\sum_{v,w\in N(t)}\tilde Y(v,w)\bigg] = \P_x\bigg[\sum_{v,w\in N(t)}\ind_{\{S(v,w)\leq s\}}\ind_{\{\zeta(v,s)>0,\zeta(w,t)>0\}}Y(v_s,w)\ind_{\{v=v_s\}}e^{\int_s^t R(X_v(u))du}\bigg]\\
+ \P_x\bigg[\sum_{v,w\in N(t)}\ind_{\{S(v,w)>t\}}\ind_{\{\zeta(v,t)>0\}} Y(v_s,w)\bigg].
\end{multline*}
First note that
\begin{align*}
\P_x\bigg[\sum_{v,w\in N(t)}\ind_{\{S(v,w)>t\}}\ind_{\{\zeta(v,t)>0\}} Y(v_s,w)\bigg] &= \Q^2_x\left[Y\ind_{\{T>t\}}\frac{1}{\zeta(\xi^2,t)} e^{\int_0^t \alpha_2(\xi^2_u)du}\right]\\
&= \Q^2_x\left[Y\frac{1}{\zeta(\xi^2,t)}e^{\int_0^t \alpha_2(\xi^2_u)du}\ind_{\{T>s\}}\Q^2_x(T>t | \Gg^2_t, \Fg^2_s)\right]\\
&= \Q^2_x\left[Y\frac{1}{\zeta(\xi^2,t)}e^{\int_0^s \alpha_2(\xi^2_u)du + \int_s^t \alpha_1(\xi^2_u)du}\ind_{\{T>s\}}\right].
\end{align*}
We will also show that
\begin{multline}\label{suffice}
\P_x\bigg[\sum_{v,w\in N(t)}\ind_{\{S(v,w)\leq s\}}\ind_{\{\zeta(v,s)>0,\zeta(w,t)>0\}}Y(v_s,w)\ind_{\{v=v_s\}}e^{\int_s^t R(X_v(u))du}\bigg]\\
= \Q^2_x\left[Y\ind_{\{T\leq s\}} \frac{\zeta(\xi^1,T)}{\zeta(\xi^1,s)\zeta(\xi^2,t)}e^{\int_0^T \alpha_2(\xi^1_u)du + \int_T^s \alpha_1(\xi^1_u)du + \int_T^t \alpha_1(\xi^2_u)du}\right].
\end{multline}
Combining these two equalities, we get that 
\[\P_x\bigg[\sum_{v,w\in N(t)}\tilde Y(v,w)\bigg] = \Q^2_x\left[Y \frac{\zeta(\xi^1,T\wedge s)}{\zeta(\xi^1,s)\zeta(\xi^2,t)}e^{\int_0^{T\wedge s} \alpha_2(\xi^1_u)du + \int_{T\wedge s}^s \alpha_1(\xi^1_u)du + \int_{T\wedge s}^t \alpha_1(\xi^2_u)du}\right]\]
which will be enough to complete the proof.

It remains to show (\ref{suffice}). By the definition (\ref{Pdef}) of $\P^2_x$,
\begin{multline*}
\P_x\bigg[\sum_{v,w\in N(t)}\ind_{\{S(v,w)\leq s\}}\ind_{\{\zeta(v,s)>0,\zeta(w,t)>0\}}Y(v_s,w)\ind_{\{v=v_s\}}e^{\int_s^t R(X_v(u))du}\bigg]\\
=\P^2_x\Bigg[\ind_{\{T\leq s\}}\ind_{\{\zeta(\xi^1,s)>0,\zeta(\xi^2,t)>0\}}Y \ind_{\{\psi^1_t=\psi^1_s\}}e^{\int_s^t R(\xi^1_u) du}\\
\cdot\bigg(\prod_{v\leq \psi^1_{T}}A_v^2\bigg)\bigg(\prod_{\psi^1_{T}<v< \psi^1_t}A_v\bigg)\bigg(\prod_{\psi^2_{T}< v < \psi^2_t}A_v\bigg)\Bigg].
\end{multline*}
On the event $\{\psi^1_t=\psi^1_s\}$, the second product above can be restricted to $v < \psi^1_s$ without changing anything. Then, using the fact that $Y$ respects the tree at time $s$, the above is
\[\P^2_x\Bigg[\ind_{\{T\leq s\}}\ind_{\{\zeta(\xi^1,s)>0,\zeta(\xi^2,t)>0\}}Y \bigg(\prod_{v\leq \psi^1_{T}}A_v^2\bigg)\bigg(\prod_{\psi^1_{T}<v< \psi^1_s}A_v\bigg)\bigg(\prod_{\psi^2_{T}< v < \psi^2_t}A_v\bigg)\Bigg].\]
Using again the fact that $Y$ respects the tree at time $s$, we see that given $\F_s^2$, on the event $\{T\leq s\}$,
\[\frac{\zeta(\xi^1,r)}{\zeta(\xi^1,s)}e^{-\int_s^r \alpha_1(\xi^1_u)du}\ind_{\{\zeta(\xi^1,r)>0\}}\prod_{\psi^1_s < v < \psi^1_r} A_v, \hs r\geq s\]
is a martingale that is independent of $Y \prod_{\psi^2_s <v< \psi^2_t}A_v$. Putting this together with the above, we have that
\begin{multline}\label{PtoP2}
\P_x\bigg[\sum_{v,w\in N(t)}\ind_{\{S(v,w)\leq s\}}\ind_{\{\zeta(v,s)>0,\zeta(w,t)>0\}}Y(v_s,w)\ind_{\{v=v_s\}}e^{\int_s^t R(X_v(u))du}\bigg]\\
= \P^2_x\Bigg[\ind_{\{T\leq s\}}\ind_{\{\zeta(\xi^1,t)>0,\zeta(\xi^2,t)>0\}}Y\frac{\zeta(\xi^1,t)}{\zeta(\xi^1,s)}e^{-\int_s^t \alpha_1(\xi^1_u)du} \bigg(\prod_{v\leq \psi^1_{T}}A_v^2\bigg)\bigg(\prod_{\psi^1_{T}<v< \psi^1_t}A_v\bigg)\bigg(\prod_{\psi^2_{T}< v < \psi^2_t}A_v\bigg)\Bigg].
\end{multline}

Recall now that on the event $\{T\leq s\}$, since $s\leq t$, we have that $T\leq t$ and therefore by Definitions \ref{zetatildedef} and \ref{Qdef},
\begin{multline*}
\left.\frac{d\Q^2_x}{d\P^2_x}\right|_{\F^2_t} = \frac{\zeta(\xi^1,t)\zeta(\xi^2,t)}{\zeta(\xi^1,T)}\ind_{\{\zeta(\xi^1,t)>0,\zeta(\xi^2,t)>0\}}e^{-\int_0^{T} \alpha_2(\xi^1_u) du - \int_{T}^t \alpha_1(\xi^1_u) du - \int_{T}^t \alpha_1(\xi^2_u)du}\\
\cdot\bigg(\prod_{v\leq \psi^1_{T}}A_v^2\bigg)\bigg(\prod_{\psi^1_{T}<v< \psi^1_t}A_v\bigg)\bigg(\prod_{\psi^2_{T}< v < \psi^2_t}A_v\bigg).
\end{multline*}
Applying this to (\ref{PtoP2}), we get that
\begin{multline*}
\P_x\bigg[\sum_{v,w\in N(t)}\ind_{\{S(v,w)\leq s\}}\ind_{\{\zeta(v,s)>0,\zeta(w,t)>0\}}Y(v_s,w)\ind_{\{v=v_s\}}e^{\int_s^t R(X_v(u))du}\bigg]\\
=\Q^2_x\left[Y\ind_{\{T\leq s\}} \frac{\zeta(\xi^1,T)}{\zeta(\xi^1,s)\zeta(\xi^2,t)}e^{\int_0^T \alpha_2(\xi^1_u)du + \int_T^s \alpha_1(\xi^1_u)du + \int_T^t \alpha_1(\xi^2_u)du}\right].
\end{multline*}
This establishes (\ref{suffice}) and completes the proof.
\end{proof}

\section{Many-to-few in discrete time}\label{discrete_sec}
We state here a version of the many-to-few lemma for discrete-time processes. We shall not prove it, as it is very similar to the continuous-time version studied above.

We begin, under a probability measure $\P_x$, with one particle in generation $0$ located at $x\in J$. Any particle at position $y$ has children whose number and positions are decided according to a finite point process $\mathcal D_y$ on $J$. The children of particles in generation $n$ make up generation $n+1$. We define $G(n)$ to be the set of all particles in generation $n$, $N(n) = \#G(n)$ to be the number of such particles, and $X_v$ to be the position of particle $v$. We set $m_j(y) = \P_y[N(1)^j]$ to be the $j$th moment of the number of particles created by the point process $\mathcal D_y$. Write $|v|$ to be the generation of particle $v$. For a particle $v$ in generation $n\geq1$, let $p(v)$ be its parent in generation $n-1$.

\subsection{The measure $\Q^k_x$ and the main result in discrete time}
We define a new measure $\P^k_x$ which has $k$ distinguished lines of descent $\psi^1,\ldots,\psi^k$ just as in the continuous-time case, which we call spines. Under $\P^k_x$, if a particle carrying $j$ marks (i.e.\ the particle is part of $j$ spines) in generation $n$ has $l$ children in generation $n+1$, then each of its $j$ marks chooses a particle to follow in generation $n+1$ uniformly at random from the $l$ children. We let $\xi^i_n$ be the position of the $i$th spine in generation $n$ and define $\skel(n)$ to be the set of all particles of generation at most $n$ which are part of at least one spine. Let $D_v$ be the number of marks carried by particle $v$.

For any $i$, we note that $X_{\xi^i_0}, X_{\xi^i_1}, X_{\xi^i_2}, \ldots$ is a Markov chain with some generator $\mathcal M'$ not depending on $i$. Suppose that $\zeta(X, n)$, $n\geq0$ is a functional of a process $(X_n, n\geq0)$ such that if $(X_n,n\geq0)$ is a Markov process with generator $\mathcal M'$ then $\zeta(X,n)$, $n\geq0$ is a martingale with respect to the natural filtration of $(X_n,n\geq0)$.

Under $\Q^k_x$ particles behave as follows:
\begin{itemize}
\item A particle at position $y$ carrying $j$ marks has children whose number and positions are decided by a point process such that:
\vspace{-1.5mm}
		\begin{itemize}
		\item for each $j$ and $l\geq0$, $\Q^j_y(N(1)=l) = l^j\P_y(N(1)=l)/\P_y[N(1)^j]$;
		\item for each $i$, the sequence $X_{\xi^i_0}, X_{\xi^i_1}, X_{\xi^i_2}, \ldots$ is a Markov chain distributed as if under the changed measure $Q^i_x|_{\Gg^{\{i\}}_n} := \zeta(\xi^i,n)\P^k_x|_{\Gg^{\{i\}}_n}$.
		\end{itemize}
\vspace{-2.5mm}
\item Given that $a$ particles $v_1,\ldots,v_a$ are born at such a branching event, the $j$ spines each choose a particle to follow independently and uniformly at random.
\vspace{-1mm}
\item Particles not in the skeleton (those carrying no marks) have children according to the point process $\mathcal D_y$ when at position $y$, just as under $\P$.
\end{itemize}
In other words, under $\Q^k_x$ spine particles move as if weighted by the martingale $\zeta$, they breed at a modified rate, and they give birth to size-biased numbers of children. The birth rate and number of children depend on how many marks the spine particle is carrying, whereas the motion does not.

\begin{lem}[Many-to-few in discrete time]\label{discrete_many_to_few}
For any $k\geq1$ and $\Fg^k_n$-measurable $Y$ such that
\[Y = \sum_{v_1,\ldots, v_k\in G(n)\cup\{\Delta\}} Y(v_1,\ldots,v_k) \ind_{\{\psi^1_n = v_1, \ldots, \psi^k_n = v_k\}}\]
we have
\begin{multline*}
\P_x\Bigg[\sum_{v_1,\ldots,v_k \in G(n)} Y(v_1,\ldots, v_k)\ind_{\{\zeta(v_i,n)>0\hsl\forall i=1,\ldots,k\}}\Bigg]\\
= \Q^k_x\Bigg[ Y \prod_{v\in \skel(n)\setminus\{\emptyset\}} \frac{\zeta(p(v),|v|-1)}{\zeta(v,|v|)}m_{D_{p(v)}}(X_{p(v)})\Bigg].
\end{multline*}
\end{lem}

\subsection*{Acknowledgements}
The authors are indebted to both Elie A\"id\'ekon and Julien Berestycki for some very helpful discussions, and also thank Pascal Maillard for checking an earlier draft.

\bibliographystyle{plain}
\def\cprime{$'$}

\end{document}